\documentclass[11pt]{amsart}
\usepackage{geometry}                
\usepackage{graphicx}
\usepackage{amssymb}
\usepackage{epstopdf}
\newtheorem{theorem}{Theorem}[section]
\newtheorem{lemma}{Lemma}[section]
\newtheorem{corollary}{Corollary}[section]
\newtheorem{proposition}{Proposition}[section]

\newtheorem{definition}{Definition}[section]

\newtheorem{remark}{Remark}[section]
\newtheorem{question}{Question}[section]
\numberwithin{equation}{section}

\def\A{\mathbb A}
\def\Z{\mathbb Z}
\def\Q{\mathbb Q}
\def\R{\mathbb R}
\def\P{\mathbb P}
\def\C{\mathbb C}
\def\A{\mathbb A}

\def\a{\alpha}
\def\b{\beta}
\def\g{\gamma}
\def\l{\lambda}

\title{VARIETIES IN CAGES:  a Little Zoo of Algebraic Geometry }
\author{Gabriel Katz}
\address{MIT, Department of Mathematics, 77 Massachusetts Ave., Cambridge, MA 02139, U.S.A.}
\email{gabkatz@gmail.com}

\begin{document}

\begin{abstract}
A $d^{\{n\}}$-{\sf cage} $\mathsf K$ is the union of $n$ groups of hyperplanes in $\mathbb P^n$, each group containing $d$ members. The hyperplanes from the distinct groups are in general position, thus producing $d^n$ points where hyperplanes from all groups intersect. These points are called the {\sf nodes} of $\mathsf K$. We study the combinatorics of nodes that impose independent conditions on the varieties $X \subset \mathbb P^n$ containing them. We prove that if $X$, given by homogeneous polynomials of degrees $\leq d$, contains the points from such a special set $\mathsf A$ of nodes, then it contains all the nodes of $\mathsf K$.  Such a variety $X$ is very special: in particular, $X$ is a complete intersection.
\end{abstract}
 
\maketitle

\section{Introduction}

This paper is an extension and generalization of \cite{K}, which dealt with algebraic curves in {\sf plane cages}, to algebraic varieties in the multidimensional cages (see Definition \ref{cage}). In this text, we use the term ``variety" as a synonym of ``algebraic set".

Our tools are mostly \emph{combinatorial}. They are based on some ``Fubini's-flavored" versions of the B\'{e}sout Theorem in the spirit of \cite{GHS}, \cite{GHS1}, and \cite{GHS2}. Although the results of this paper fit well into the general framework of Cayley-Bacharach theorems (see \cite{EGH}, Theorem CB6), to apply this general machinery to the very special configurations of nodes from a given cage still requires some effort. So here we present a more direct and elementary argument, whose application is limited in scope, but geometrically transparent.     
\smallskip

Let us consider two groups of lines in the plane (real or complex), each group comprising three lines. We call such a configuration $\mathsf K$ of six lines a $3\times 3$-{\sf cage}, or $3^{\{2\}}$-{\sf cage} for short. We label the lines of the first group with \emph{red}, and of the second group with \emph{blue}. Assume that there are exactly $9$ points where the blue lines intersect the red lines. We call them the  {\sf nodes} of the cage.
\smallskip

Our original motivation for studying the varieties in cages comes from the following classical result in the theory of plane cubic curves, a special case of Chasles' Theorem for  a pair of cubic curves that meet at $9$ points \cite{Ch}.

\begin{theorem}{\bf (The Cage Theorem for Plane Cubics)}\label{3x3}  
Any plane cubic curve $\mathcal C$, passing through eight nodes of a $3\times 3$-cage, will automatically pass through the ninth node.
\end{theorem}

Recall that Cage Theorem \ref{3x3} reflects the {\sf associativity} of the binary group operation ``$+$" on an elliptic curve $\mathcal C$ (see Figure \ref{asso}). 

\begin{figure}
  \includegraphics[width=0.5\textwidth]{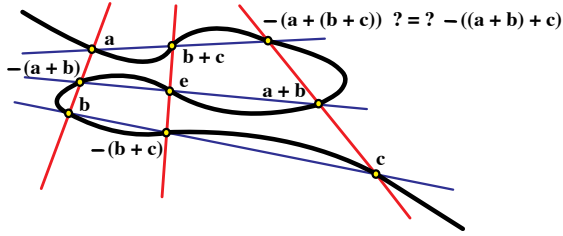}
  \caption{\small{Cage Theorem \ref{3x3} helps to validate the associativity of the group operation on a nonsingular cubic curve; $e$ denotes the neutral element.}}
  \label{asso}
\end{figure}

\begin{definition}\label{cage}
 A $d^{\{n\}}$-{\sf cage} $\mathsf K$ is a configuration of $n$ distinctly colored groups of $d$ hyperplanes each (the entire hyperplane configuration $\mathsf K$ consists of $nd$ hyperplanes) located in the $n$-space (projective or affine) in such a way, that $\mathsf K$ generates exactly $d^n$ points where the hyperplanes of all $n$ distinct colors $\a_1, \dots, \a_n$ intersect transversally. (It follows that any $n$-tuple of distinctly colored hyperplanes are in general position in the ambient $n$-space.) These points are called the {\sf nodes} of the cage. \hfill $\diamondsuit$
\end{definition}

Hyperplanes in $\P^n$ form a dual projective space $\P^{n \ast}$, the space of linear homogeneous functions, considered up to proportionality. Therefore $d$ hyperplanes of the same color from a $d^{\{n\}}$-cage $\mathsf K \subset \P^n$ represent an \emph{unordered} configuration of $d$ points in $\P^{n \ast}$, a point in the symmetric product $\mathsf{Sym}^d(\P^{n \ast})$. So the color-ordered collection of $n$ such points from $\mathsf{Sym}^d(\P^{n \ast})$ is a point of the space $\big(\mathsf{Sym}^d(\P^{n \ast})\big)^n$. By Definition \ref{cage}, any set of $n$ hyperplanes of \emph{distinct} colors has a single intersection point. The requirement that some set of $n$ hyperplanes of distinct colors has multiple intersection points in $\P^n$ puts algebraic constraints on the coefficients of the $dn$ homogeneous linear polynomials (in $n+1$ variables) that define the hyperplanes. Similarly, the requirement in Definition \ref{cage} that all transversal $n$-colored intersections are distinct, and thus numbering $d^n$, produces a Zariski open set. Therefore, we get:

\begin{lemma}\label{space_of_cages} The $d^{\{n\}}$-cages form a Zariski open set $\mathcal K$ in the $(dn^2)$-dimensional space $\big(\mathsf{Sym}^d(\P^{n \ast})\big)^n$. The group of projective transformations $\mathsf{PGL}_\A(n+1)$ acts naturally on $\big(\mathsf{Sym}^d(\P^{n \ast})\big)^n$, and thus on the set $\mathcal K$. \hfill $\diamondsuit$
\end{lemma}

The problem we address in this paper is to describe the varieties that contain all $d^n$ nodes of a given cage $\mathsf K$. It turns out, that every variety $V$, defined by polynomials of degrees $\leq d$ and containing the node set $\mathsf N$,  is very special indeed. In particular, $V$ must be a {\sf complete intersection} (see Definition \ref{complete_int}) of the type $(\underbrace{d, \dots , d}_{s})$, where $s = n- \dim V$. Furthermore, the requirements that a hypersurface of degree $\leq d$ will pass through the nodes of a $d^{\{n\}}$-cage are very much \emph{redundant}. In this article, we describe the combinatorics of the nodes that impose \emph{independent} constraints on the hypersurface in question. We call such maximal set $\mathsf A$ of ``independent" nodes {\sf supra-simplicial} (see Definition \ref{def2.1} and Figure \ref{supra}). Crudely, the proportion of cardinalities $\frac{\#\mathsf A}{\# \mathsf N}$ declines as $\sim 1/n!$ with the growth of $d$.  
\smallskip

We have mentioned already that some of our results share the flavor with a remarkable family of classical theorems of Algebraic Geometry (see \cite{C}, \cite{C1}, \cite{DGO},  \cite{EGH}, and Theorem \ref{Bach} below). These classical theorems operate within a much less restrictive environment than the one of the $d^{\{n\}}$-cages. However, the theorems about varieties in cages are more geometrical, transparent, and easy to state. 

To provide a point of reference, let us describe briefly this family of classical results, known under the name ``{\sf Cayley-Bacharach theorems}".

Let $\Z_+$ denote non-negative integers. Recall that the {\sf Hilbert functions} $h_X: \Z_+ \to \Z_+$ of a variety $X$ over a field $\A$ associates with a non-negative integer $k$ the dimension of the $k$-graded portion of the quotient ring $\A[x_0, \dots, x_N] / \mathcal I_X$, where $\mathcal I_X$ denotes the zero ideal  that defines $X$.  The ring $\A[x_0, \dots, x_N] / \mathcal I_X$ is called the {\sf coordinate ring} of $X$.\smallskip

Since the node set $\mathsf N$ of a $(d\times d)$-cage is the intersection locus of $d$ red and $d$ blue lines, Theorem \ref{3x3} is a special case of the Cayley-Bacharach Theorem (\cite{C}, \cite{B}), stated below.  For a complete intersection $X\subset \P^2$, Theorem \ref{Bach} connects the Hilbert functions $h_X: \Z_+ \to \Z_+$,  $h_{X_1}: \Z_+ \to \Z_+$, and $h_{X_2}: \Z_+ \to \Z_+$ of a finite set $X$, its subset $X_1$, and its complement $X_2:= X \setminus X_1$. Recall that, for a $0$-dimensional variety $X$ and all sufficiently big $k$, $h_X(k) = |X|$, the cardinality of $X$. 

\begin{theorem}{\bf (Cayley-Bacharach)}\label{Bach}
Let $\mathcal D$ and $\mathcal E$ be two projective plane curves of degrees $d$ and $e$, respectively, and let the finite set $X = \mathcal D \cap \mathcal E$ be a complete intersection in $\P^2$. Assume that $X$ is the disjoint union of two subsets, $X_1$ and $X_2$. Then for any $k \leq d+e-3$, the Hilbert functions  $h_X$, $h_{X_1}$, and $h_{X_2}$ are related by the formula:
$$h_X(k) - h_{X_1}(k) = |X_2| - h_{X_2}((d + e - 3) - k).  \qquad \qquad \qquad \qquad \hfill \diamondsuit $$ 
\end{theorem}

The RHS of this formula describes the failure to impose independent constrains by the points of the set $X_2$ on the polynomials of degree $k$. So Cayley-Bacharach Theorem may be viewed as a duality claim; however, it does not compute explicitly each of the dual quantities in the formula above. In contrast, our Theorem \ref{main_th} accomplishes this task, provided $X = X_1 \coprod X_2$ are very special: namely, $X$ is the $0$-dimensional variety of nodes of a cage $\mathsf K$, and $X_1 \subset X$ is a supra-simplicial set. \smallskip    

Theorem \ref{Bach} admits a comprehensive generalization by Davis, Geramita, Orecchia \cite{DGO}, and by Geramita, Harita, Shin (see \cite{GHS1}, and especially \cite{GHS2}, Theorem 3.13). It is a ``Fubini-type" theorem for the Hilbert function of a finite subset $X \subset \P^n$ that is contained in the union of a family of hypersurfaces $\{H_i \}_{1 \leq i \leq s}$, whose degrees $\{d_i\}$ add up to the degree of $X$. Under some subtle hypotheses that regulate the interaction between $X$ and the hypersurfaces $\{H_i \}_{1 \leq i \leq s}$ (they include the hypotheses ``$X = \coprod_i (X \cap H_i)"$), a nice formula for the Hilbert functions $\{h_{X \cap H_i}: \Z_+ \to \Z\}_{1 \leq i \leq s}$ of $H_i$-slices of $X$ emerges:
$$h_X(k) = h_{X \cap H_1}(k) + h_{X\cap H_2}(k - d_1)+ \dots +h_{X\cap H_s}(k- (d_1 + \dots +d_{s-1})).$$
A clear beautiful overview of the research, centered on the Cayley-Bacharach type theorems, can be found in \cite{EGH}.
\smallskip

Now let us describe \emph{the results of the paper and its structure} in some detail.  The paper is divided in two sections, including the Introduction. \smallskip

Our main results of are: Theorem \ref{main_th}, Theorem \ref{smooth}, and Corollary \ref{unique_hyper}. Here is a summary of their claims. Any variety $X \subset \P^n$ that is the zero set of homogeneous polynomials of degrees $\leq d$ and contains a supra-simplicial set $\mathsf A$ of nodes of a given $d^{\{n\}}$-cage $\mathsf K  \subset \P^n$ contains all the nodes of  $\mathsf K$. Such $X$ is a complete intersection of the multi-degree $(\underbrace{d, \dots , d}_{s})$, where $s = \textup{codim}(X, \P^n)$. Moreover, $X$ is smooth in the vicinity of the node set $\mathsf N$. The variety $X$ is completely determined by $\mathsf A$ and the tangent to $X$ space $\tau_p$ at any of the nodes $p$. Conversely, any subspace $\tau_p \subset T_p(\P^n)$ of codimension $s$, where $p \in \mathsf N$, with the help of $\mathsf A$, produces such a variety $X$.\smallskip

In all the figures, we restrict ourselves to depictions of cages in the space $\R^3$. Most of the figures are produced with the help of the \emph{Graphing Calculator} application. In the figures, for technical reason, the nodes of the cages are invisible. Although the images depict real surfaces in only $3^{\{3\}}$- and $4^{\{3\}}$-cages, the entire exhibition looks surprisingly rich.  \smallskip

We tried to make this text friendly to readers who, as the author himself, are not practitioners of Algebraic Geometry but who may enjoy a visit to the small zoo of varieties in cages, a microcosmos of the old fashion Italian style Algebraic Geometry.

\section{A Multidimensional Zoo}
As a default, we choose the base field $\A$ to be the field of real or complex numbers. 
In the notations below, we do not emphasize the dependence of our constructions on the choice of a base field.\smallskip

Let $\mathcal L_j$ be a degree $d$ homogeneous polynomial whose zero set is the union of $d$ hyperplanes of a particular color $\a_j$ ($\mathcal L_j$ is a product of $d$ linear forms).  Since $\deg(\mathcal L_j) = d$, B\'{e}zout's Theorem implies that the solution set $\mathsf N$ of the system $\{\mathcal L_j =0\}_{j \in [1, n]}$ consists of $d^n$ points at most, provided that  $\mathsf N$ is finite. Thus Definition \ref{cage} implies that each node $p \in \mathsf N$ of the cage belongs to a single hyperplane of a given color and the hyperplanes of distinct colors are in general position at $p$, and thus in the ambient $n$-space.  It follows that the node locus $\mathsf N \subset \P^n$ is a $0$-dimensional complete intersection of degree $d^n$.
\smallskip

\noindent{\bf Example 2.1.} Consider the complex Fermat curve $\mathcal F \subset \C\P^2$, given by the equation $\{\tilde x^d + \tilde y^d = \tilde z^d\}$ in the homogeneous coordinates $[\tilde x : \tilde y : \tilde z]$. In the affine coordinates $(x, y) = (\tilde x/\tilde z, \tilde y/\tilde z)$, its equation may be written as $x^d + y^d = 1$, or as $\prod_\xi (x - \xi) + \prod_\eta (y- \eta) = 0$, where $\xi, \eta$ run over the set of complex $d$-roots  $\{\sqrt[d]{-1/2}\}$. Therefore $\mathcal F$ passes trough the nodes of the $d \times d$-cage $\mathsf K := (\bigcup_\xi \{x = \xi\})\,\bigcup  \,(\bigcup_\eta \{y = \eta\}) \subset \C^2$.
\hfill $\diamondsuit$
\smallskip 

Let $\mathbf I^n(d)$ be the subset $\{I = (i_1, i_2, \dots,  i_n)\}$ of the lattice $\Z^n_+$, such that each $i_j \in [1, d]$. So $\mathbf I^n(d)$ is a $n$-dimensional ``cube" of the size $d$. By definition, $\| I\| = \sum_{j =1}^n i_j$.\smallskip

If we introduce \emph{some order} among the hyperplanes of the same color $\a_j$ ($j = 1, \dots , n$), then each node $p_I$ of $\mathsf K$ will be marked with a unique multi-index $I \in  \mathbf I^n(d)$.

\begin{definition}\label{def2.1} A set of nodes $\mathsf T$ from $d^{\{n\}}$-cage $\mathsf K$ is called {\sf simplicial} if, with respect to some orderings of the hyperplanes in each group, it is comprised of the nodes $\{p_I\}_{I \in \mathbf I^n(d)}$, subject to the constraints $\| I\|  \leq d + 1$. \smallskip

A set of nodes $\mathsf A$ from a cage $\mathsf K$ is called {\sf supra-simplicial} if, with respect to some orderings of the hyperplanes in each group, it is comprised of the nodes $\{p_I\}_{I \in \mathbf I^n(d)}$, subject to the constraints $\| I\|  \leq d + 2$. (see Figure \ref{supra}, where the grid corner is located at $(1, 1, 1)$). \hfill $\diamondsuit$
\end{definition}

\begin{figure}
  \includegraphics[width=0.45\textwidth]{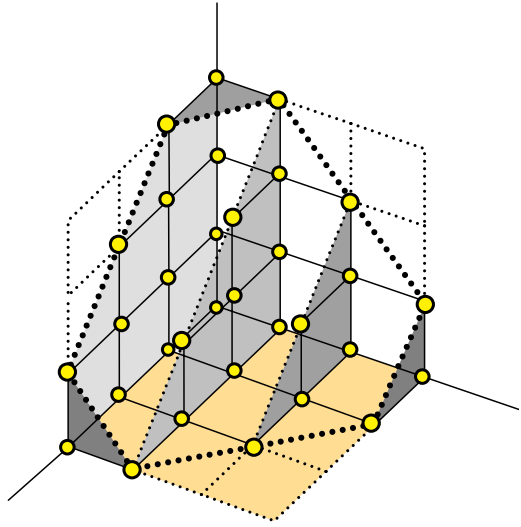}
  \caption{\small{A supra-simplicial set of nodes in a $4^{\{3\}}$-cage}}
  \label{supra}
\end{figure}

\noindent{\bf Example 2.2.} For $d =2$, the $2^{\{n\}}$-cage is modeled after the union of the hyperplanes in $\R^n$ that extend the faces of a $n$-cube. The cardinality of the node locus $\mathsf N$ is $2^n$, the cardinality of the simplicial set $\mathsf T$ is $n+1$, while the cardinality of the supra-simplicial set $\mathsf A$ is $C_n^2 + n +1 = \frac{1}{2}(n^2 + n +2)$.
\hfill $\diamondsuit$
\smallskip

\noindent{\bf Example 2.3.} A famous example of a $K3$-surface is given by the equation $\{y_0^4 + y_1^4+ y_2^4 + y_3^4= 0\}$ in $\C\P^3$, or by the equation $\{x_1^4 + x_2^4+ x_3^4 +1 = 0\}$ in $\C^3$. Using the partition  $\{1 = 1/3 +1/3 +1/3\}$, the latter equation may be written in the form $$\prod_\a(x_1- \a)+ \prod_\b(x_2- \b) + \prod_\g(x_3 -\g) = 0,$$ where $\a, \b, \g$ each runs over the four complex roots of the equation $\{z^4 = -1/3\}$. Therefore, the $K3$-surface contains all the $64$ nodes of a $4^{\{3\}}$-cage $\mathsf K$, defined by the equations $$\big\{\prod_\a(y_1- \a \cdot y_0) =0\big\} \bigcup \big\{\prod_\b(y_2- \b\cdot  y_0) =0\big\} \bigcup \big\{\prod_\g(y_3 -\g\cdot  y_0) = 0\big\}.$$

\noindent In fact, the $K3$-surface is nailed to the notes of a 2-dimensional variety of cages, produced in similar ways by writing down $1$ as a sum of three complex numbers, all different from $0$. The previous construction was based on the composition  $\{1 = 1/3 +1/3 +1/3\}$. 

We notice that the nodes of this cage $\mathsf K$ are ``invisible" in $\R\P^3$. 
\smallskip

The permutation group $\mathsf S_4$ of order $24$ acts on $\C\P^3$ by permuting the coordinates $(y_0,  y_1, y_2, y_3)$. Under this $\mathsf S_4$-action, this surface is invariant. In contrast, the cage $\mathsf K$ is invariant only under the $\mathsf S_3$-action that permutes the coordinates $(y_1, y_2, y_3)$. (This action does not preserve the colors of the cage!) Thus, using the $\mathsf S_4$-action on $\mathsf K$, the $K3$-surface contains the nodes of at least \emph{four} distinct $4^{\{3\}}$-cages in $\C\P^3$.
\hfill $\diamondsuit$
\smallskip

\noindent{\bf Example 2.4.} Recall a remarkable Cayley-Salmon Theorem \cite{C1}: any smooth complex cubic surface $X$ contains exactly $27$ lines. If $X \subset \C\P^3$ is given by the equation $\{z_0^3 + z_1^3 + z_2^3 + z_3^3  = 0\}$ (this surface is called Fermat cubic surface), then putting $\omega := e^{2\pi \mathbf i/3}$, each of these $27$ lines is given by $2$ linear constraints (see \cite{M}, Corollary (8.20)):
\begin{eqnarray}
\{z_0 + \omega^i z_1 = 0, \; z_2 + \omega^j z_3 =0\}, \;\; i, j \in [0, 2], \nonumber \\
\{z_0 + \omega^i z_2 = 0, \;  z_1 + \omega^j z_3 =0\}, \;\; i, j \in [0, 2], \nonumber \\
\{z_0 + \omega^i z_3 = 0, \;  z_1 + \omega^j z_2 =0\},  \;\; i, j \in [0, 2].
\end{eqnarray}

As in the previous examples, using the composition $\{1 = 1/3 +1/3+ 1/3\}$, we notice that $X$ is inscribed in a $3^{\{3\}}$-cage $\mathsf K$, given by the formula 
\hfill \break $$\bigcup_{j =1}^3  \Big\{\prod_{k=0}^2 \big(z_j +  \frac{1}{\sqrt[3]{3}}\, \omega^k \, z_0\big) = 0\Big\}.$$

As in Example 2.3, there exists a $2$-parameter family of cages in which $X$ is inscribed (it corresponds to different ways one can represent $1$ as a sum of three non-vanishing complex numbers). 
\smallskip

The symmetric group $\mathsf S_4$ acts on the Fermat surface $X$ by permuting the coordinates in $\C\P^3$. This action must preserve the configuration of $27$ lines in $X$ since these lines are the only ones residing in $X$.
The subgroup $\mathsf S_3 \subset \mathsf S_4$ that permutes the coordinates $(z_1, z_2, z_3)$ evidently preserves the cage $\mathsf K$, but not its colors. Thus $X$ contains the nodes of at least $4$ distinct cages in $\C\P^3$, obtained from $\mathsf K$ by the $\mathsf S_4$-action. 
\smallskip

Consider the $27$ lines, contained the $3^{\{3\}}$-cage $\mathsf K$, where two planes of distinct colors intersect (this locus is the ``$1$-skeleton" of $\mathsf K$), and compare them with the $27$ lines on a smooth cubic surface $X$ (see \cite{H}, Chapter V, Section 4, for the explicit description of the configuration the $27$ lines on $X$).
\hfill $\diamondsuit$

\begin{question} \emph{For a smooth complex cubic surface $X \subset \C\P^3$ that contains all the nodes of a given $3^{\{3\}}$-cage $\mathsf K$, how to describe in terms of $\mathsf K$ the pattern of $27$ lines that belong to $X$? Is there anything special about the locus where the $27$ lines in $X$ hit the nine planes that form the cage?} \smallskip

\emph{Perhaps, within the family of cubic surfaces $X$ that are inscribed in $\mathsf K$, the $27$ bicolored lines of the cage are ``the limits" of $27$ lines on $X$, as $X$ degenerates  into the completely reducible variety of $3$ planes of a particular color?}
 \hfill $\diamondsuit$
\end{question}

\begin{figure}
  \includegraphics[width=0.6\textwidth]{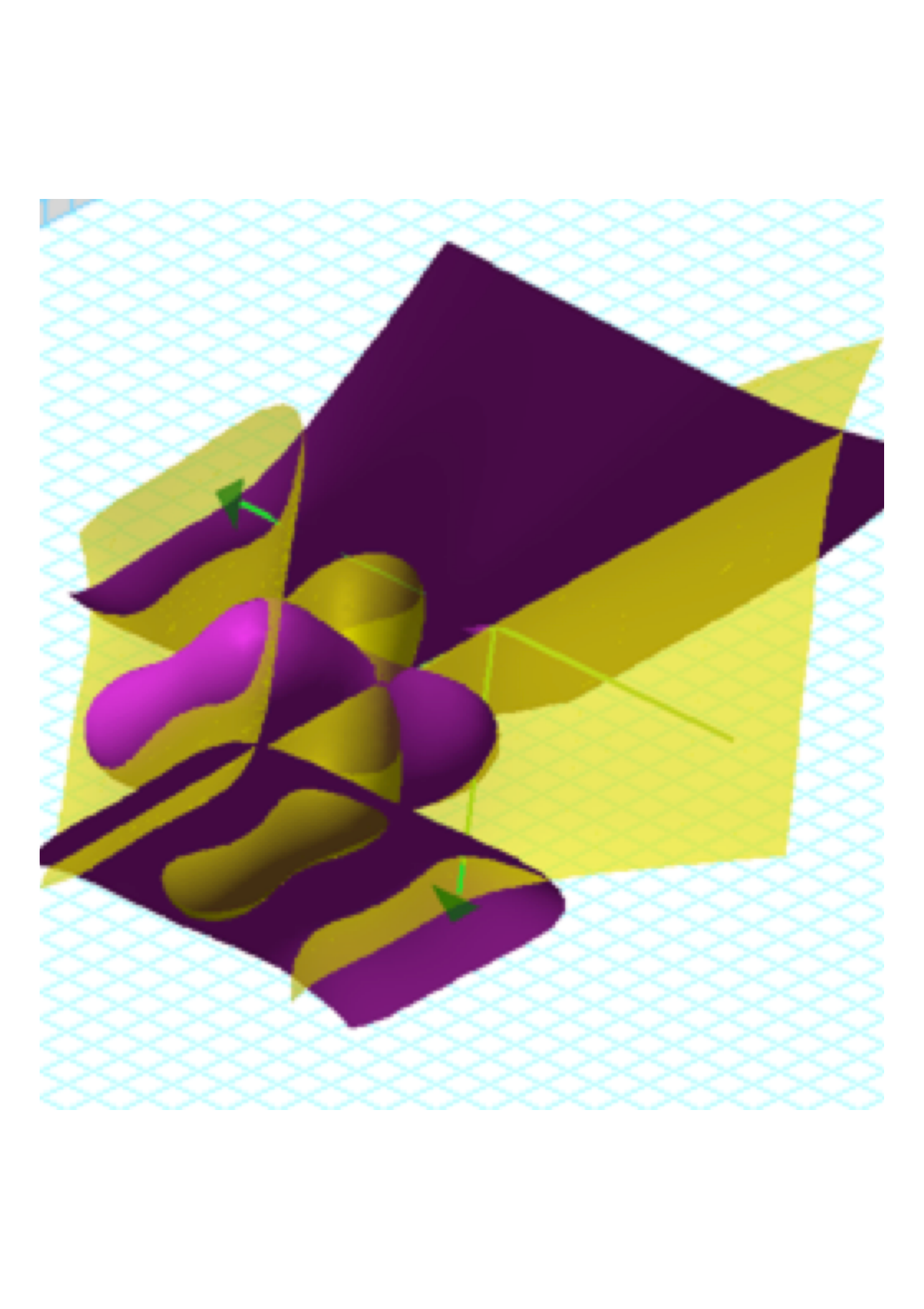}
  \caption{\small{Two cubic surfaces, each passing through the nodes of a $3^{\{3\}}$-cage $\mathsf K \subset \R^3$. Note the curve $C \subset \R^3$ of the multi-degree $(3,3)$, where the two surfaces (of two colors) intersect. $C$ also contains all the $27$ nodes of $\mathsf K$.}}
  \label{3X3}
 \end{figure}

\begin{figure}
  \includegraphics[width=0.6\textwidth]{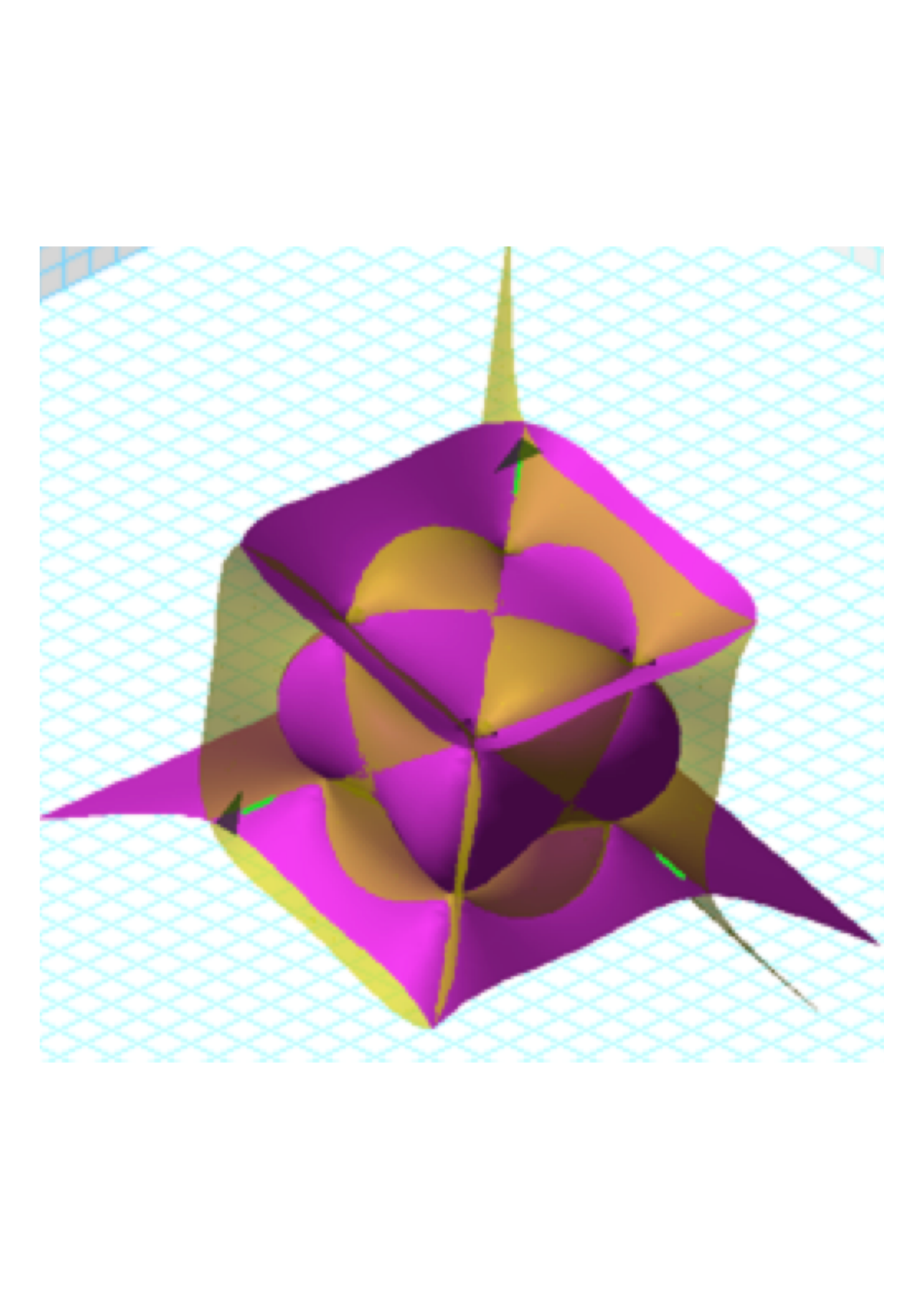}
  \caption{\small{Another pair of cubic surfaces, passing trough the $27$ nodes of the same cage $\mathsf K$, as in Figure \ref{3X3}. Again, the ``bicolored" intersection locus $C$ of the two surfaces contains all $27$ nodes of $\mathsf K$.}}
  \label{3X3A}
\end{figure}

\begin{figure}
  \includegraphics[width=0.7\textwidth]{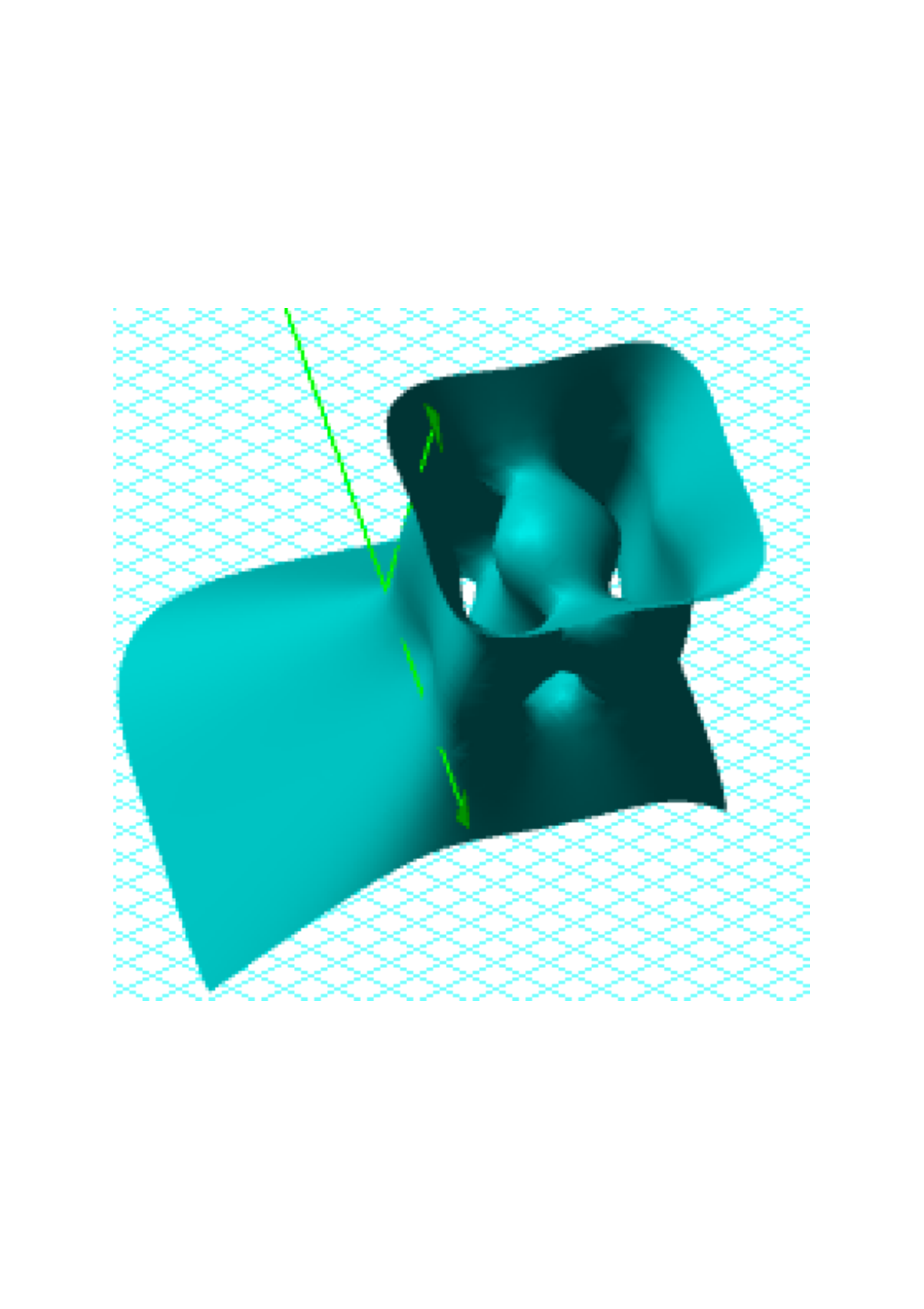}
  \caption{\small{A surface of degree $4$, passing trough $64$ nodes of a $\mathsf S_3$-symmetric $4^{\{3\}}$-cage $\mathsf K \subset \R^3$. Although this surface is not compact in $\R^3$, there are compact real surfaces of degree $4$ that pass trough the nodes of $\mathsf K$.}}
  \label{4X4A}
\end{figure}

\smallskip

By examining the diagonal lines in the Pascal Triangle, we get the following useful combinatorial fact.

\begin{lemma}\label{cardinality} Each simplicial set of nodes $\mathsf T$ in a $d^{\{n\}}$-cage is of the cardinality $C_{d+n-1}^n$.

Each supra-simplicial set of nodes $\mathsf A$ in a $d^{\{n\}}$-cage is of cardinality $C_{d+n}^n -  n$. 
\hfill $\diamondsuit$
\end{lemma}
\smallskip

Let $H_{j, i}$ be the $i$-th hyperplane of the color $\a_j$, and let  $L_{j, i}$ be a homogeneous linear polynomial in the  coordinates $(y_0, y_1, \dots , y_n)$ on the space $\A^{n+1}$ that defines $H_{j, i}$. Each $L_{j, i}$ is determined, up to proportionality, by $H_{j, i}$. In what follows, we fix particular linear forms $\{L_{j, i}\}_{i,j}$. Put $\mathcal L_j := \prod_{i \in [1, d]} L_{j, i}$. 

For any nonzero vector $\vec \l = (\l_1, \dots, \l_n) \in \A^n$, we consider the homogeneous polynomial of degree $d$ 
\begin{eqnarray}\label{P_Lambda}
\mathcal P_{\mathsf K,\, \vec\l} \; := \sum_{j \in [1, n]} \l_j \cdot \mathcal L_j.
\end{eqnarray}

Evidently, each polynomial $\mathcal P_{\mathsf K,\, \vec\l}$ vanishes at all the nodes of the cage $\mathsf K$. 
\smallskip

\begin{theorem}\label{main_th} Consider a subvariety $V \subset \P^n$, given by one or several homogeneous polynomial equations of degrees $\leq d$. 
\begin{itemize}
\item If $V$ contains all the nodes from a supra-simplicial set $\mathsf A$ of a $d^{\{n\}}$-cage $\mathsf K \subset \P^n$, then $V$ contains all $d^n$ nodes of the cage. Moreover, any such variety $V$ is given by polynomial equations of the form 
$\{\mathcal P_{\mathsf K, \vec\l} = 0\}_{\vec\l}$  for an appropriate choice of vectors $\vec\l$ (see (\ref{P_Lambda})).
\smallskip

\item In contrast, no such variety $V$ contains all the nodes from a simplicial set $\tilde{\mathsf T}$ of any \hfill\break $(d+1)^{\{n\}}$-cage $\tilde{\mathsf K} \subset \P^n$.
\end{itemize}
\end{theorem}

\begin{proof} As in the case of encaged plane curves \cite{K}, the argument is based on a combinatorial similarity between the Newton's diagram of a generic polynomial of degree $d$ in $n$ variables and a simplicial set $\tilde{\mathsf T}$ of nodes of any $(d+1)^{\{n\}}$-cage. Also the cardinality of such a Newton's diagram exceeds the cardinality of a supra-simplicial set $\mathsf A$ of nodes of a $d^{\{n\}}$-cage $\mathsf K$ by $n$. In other words, the dimension of the variety of hypersurfaces of degree $d$ in the space $\P^n$ exceeds $\#\mathsf A$ by $n-1$. Indeed, the
monomials in the affine variables $x_1, \dots, x_n$ of degree $\leq d$ (equivalently, the homogeneous monomials in the variables $y_0, \dots, y_n$ of degree $d$) are in one-to-one correspondence with the set $\mathsf B$ none-negative integral $n$-tuples $I \in \Z^n$, subject to the inequality $\|I\| \leq d$. At the same time, the nodes $\{p_I\}$ of an supra-simplicial set $\mathsf A$ satisfy the inequality $\| I\| \leq d + 2$ together with $\{1 \leq i_s \leq d\}_{s \in [1, n]}$.
Shifting  by the vector $(-1, \dots , -1)$ embeds $\mathsf A$ into $\mathsf B$ so that only the $n$ corners $(d, 0, \dots , 0), (0, d, \dots, 0), \dots (0, 0, \dots, d)$ 
of the Newton diagram remain outside of the shifted $\mathsf A$. Finally, proportional polynomials define the same hypersurface.  \smallskip

The following proof is recursive in nature. The induction is carried  in $n$, the dimension of the cage. We assume that the first assertion of the theorem is valid for all $d^{\{k\}}$-cages of any size $d$ in spaces of dimension  $k < n$, and the second assertion is valid for all cages of any size $d+1$ in spaces of dimension  $k < n$. 

Our argument relies on slicing $\mathsf K \supset \mathsf A$ by the hyperplanes $\{H_{1, i} = 0\}_{i \in [1, d]}$ of the first color $\a_1$, thus reducing the argument to families of cages in $(n-1)$-dimensional affine or projective spaces. This leads to a "Fubini-type cage theorem" in the spirit of \cite{GHS2} (see Figure \ref{supra} for guidance). \smallskip 

For any integer $s \in [1, d-1]$, we consider the $(d-s+1)^{\{n-1\}}$ sub-cage $\mathsf K^{[s]} \subset \mathsf K\, \cap \, H_{1, s} $, formed by the hyperplanes $$H_{1, s}\, \bigcap \, \big(\bigcup_{j \in [2,\,n],\; i \in [1,\, d-s+1]} H_{j, i}\big)$$ in $H_{1, s} \approx \P^{n-1}$. 
In the hyperplane $H_{1, s}$, the cage $\mathsf K^{[s]}$ is given by the equation $$\big\{\mathcal L\mathcal L^{[s]} \, := \prod_{j \in [2,\, n],\; i \in [1,\, d-s+1]} L_{j, i} \;= \;0\big\}.$$ 

We denote by $\mathsf T^{[s]}$ the simplicial set of nodes in $\mathsf T \cap \mathsf K^{[s]}$ and by $\mathsf A^{[s]}$---the set of nodes from the supra-simplicial set $\mathsf A \cap  \mathsf K^{[s]}$. Note that the set $\mathsf T^{[s]}$ can serve as a simplicial set and $\mathsf A^{[s]}$--- as a supra-simplicial set for the cage $\mathsf K^{[s]}$.
\smallskip

We start with a given homogeneous degree $d$ polynomial $P$ in the projective coordinates $[y_0: y_1: \dots y_n]$, which vanishes at all the nodes of a supra-simplicial set $\mathsf A$ of a $d^{\{n\}}$-cage $\mathsf K \subset \P^n$.

Consider the restriction of  $P$ to the first hyperplane $H_{1,1}$ of the color $\a_1$. 
Then $P$ vanishes at the supra-simplicial set $\mathsf A^{[1]} := \mathsf A \cap H_{1,1}$ of the induced $d^{\{n-1\}}$-cage $\mathsf K^{[1]} := \mathsf K \cap H_{1,1}$, the zero set of the polynomial $\mathcal L_2 \cdot \mathcal L_3 \cdot\; \dots \; \cdot \mathcal L_n$ in $H_{1, 1}$. By induction on $n$, the restriction $P|_{H_{1,1}}$ must be of the form $\mathcal P_1 := \sum_{j \in [2, n]} \l_j^{[1]} \cdot \mathcal L_j$ (being restricted to $H_{1,1}$) for some choice of the coefficients $\l_2^{[1]}, \dots \l_n^{[1]}$. For this special choice of $(\l_2^{[1]}, \dots \l_n^{[1]})$, the difference $P -  \mathcal P_1$ is identically zero on $H_{1,1}$. If a homogeneous polynomial $R$ vanishes on a hyperplane, given by a homogeneous linear polynomial $L$, then $R$ is divisible by $L$. Indeed, by a linear change of the homogeneous coordinates, we may choose $L$ as the first new variable and write down $R$ as $L\cdot Q + S$, where $S$ is a homogeneous polynomial that depends only on the rest of the new variables.  Since $S = S|_{\{L=0\}} = 0$, the polynomial $S = 0$ identically, and $R = L\cdot Q$.

Therefore $P -  \mathcal P_1$ is divisible by the liner polynomial $L_{1,1}$. So $P = \mathcal P_1 + L_{1,1}\cdot P_1$, where $P_1$ is a homogeneous polynomial of degree $d-1$.\smallskip

Next, we consider the restrictions of $P$ and $P_1$ to the hyperplane $H_{1, 2} = \{L_{1,2} = 0\}$ of color $\a_1$. Since both $P$ and $\mathcal P_1$ vanish at the set $\mathsf A \cap H_{1, 2}$ and, by Definition \ref{cage}, $L_{1,1} \neq 0$ at the points of $\mathsf A \cap H_{1, 2}$, we conclude that $P_1$ (of degree $d-1$) must vanish at the set $\mathsf A \cap H_{1, 2}$ as well. Note that $\mathsf A \cap H_{1, 2} = \mathsf A^{[2]}$ is a \emph{simplicial set} for the induced $d^{\{n-1\}}$-cage $\mathsf K^{[2]} \subset \mathsf K \cap H_{1,2}$. So by induction, any homogeneous polynomial of degree $d-1$ that vanishes at a simplicial set $\mathsf A^{[2]}$ of the $d^{\{n-1\}}$-cage $\mathsf K^{[2]}$ must vanish at $H_{1,2}$. Hence $P_1 = L_{1,2} \cdot P_2$ for some homogeneous polynomial $P_2$ of degree $d-2$. So we get $P = \mathcal P_1 + L_{1,1}\cdot L_{1,2} \cdot P_2$.

Similarly, we argue that of $P_2$ of degree $d-2$ vanishes on the simplicial set $\mathsf A^{[3]} \subset \mathsf A \cap H_{1,3}$ of the $(d-1)^{\{n-1\}}$-cage $\mathsf A^{[3]}$. Therefore $P_2|_{H_{1,3}}$ is zero, and $P_2 = L_{1,3}\cdot P_3$ for a homogeneous  polynomial $P_3$ of degree $d-3$. As a result, $P = \mathcal P_1 + L_{1,1}\cdot L_{1,2} \cdot L_{1,3}\cdot P_3$.

Continuing this reasoning, we get eventually 
$$P =  \mathcal P_1 +  \l(L_{1,1}\cdot L_{1,2}\cdot  \dots \cdot L_{1,n}) = \sum_{j \in [2, n]} \l_j^{[1]} \cdot \mathcal L_j +   \l \mathcal L_1,$$
where $\l$ is a constant.
Therefore,
$P = \l\cdot \mathcal L_1 + \sum_{j \in [2, n]} \l_j^{[1]} \cdot \mathcal L_j $
is of the form $\mathcal P_{\mathsf K, \vec\l}$ and must vanish at every node of the $d^{\{n\}}$-cage $\mathsf K \subset \P^n$.
\smallskip

By a similar reasoning, we will validate the second claim of the theorem.
So we take any polynomial $P$ of degree $d$ that vanishes at a simplicial set $\tilde{\mathsf T}$ of a $(d+1)^{\{n\}}$-cage $\tilde{\mathsf K} \subset \P^n$. As before, we slice $\tilde{\mathsf K}$ by the hyperplanes $\{H_{1, s }\}_{i \in [1, d+1]}$ of the color $\a_1$. Now all the slices $\tilde{\mathsf T}^{[s]}$ (including the first one) are simplicial sets on $\tilde{\mathsf K}^{[s]}$. The latter locus $\tilde{\mathsf K}^{[s]}$ is given by the equations 
$$\big\{\tilde{\mathcal L\mathcal L}^{[s]} \, := \prod_{j \in [2 ,n],\; i \in [1,\, d-s+2]} L_{j, i} \;= \;0\big\}.$$ 

 Since $P$ vanishes at $\tilde{\mathsf T}^{[1]}$, by the induction hypotheses, $P|_{H_{1,1}} = 0$. This implies that $P = L_{1,1}\cdot P_1$, where $P_1$ is a homogeneous polynomial of degree $d_1$. The set $\tilde{\mathsf T}^{[2]}$ is simplicial in the cage $d^{\{n-1\}}$-cage. Since $L_{1,1}|_{\tilde{\mathsf T}^{[2]}} \neq 0$, we get that $P_1$ must vanish at the nodes from $\tilde{\mathsf T}^{[2]}$. By induction, this implies that $P_1|_{H_{1,1}} = 0$ and thus is divisible by $L_{1,1}$. So $P= L_{1,1}\cdot L_{1,2}\cdot P_2$ for a homogeneous polynomial $P_2$ of degree $d-2$. Continuing this process, we get $P= L_{1,1}\cdot L_{1,2}\cdot\, \dots, \cdot  L_{1,d} \cdot \l$ must vanish at the unique node of the set $\tilde{\mathsf T}^{[d+1]}$. This forces $\l = 0$, and so $P$ is identically zero. 
\smallskip

Finally, the validity of the basis of induction ``$n=1$" is obvious for univariate polynomials of any degree $d$. In fact, Theorem \ref{main_th} has been proven in \cite{K} for $n=2$.
\smallskip

Since the varieties $V$ we consider in the theorem are defined by polynomials of degrees $\leq d$, the claim follows.
\end{proof}

\noindent{\bf Remark 2.1.}  Note that the assumption that $\mathsf A$ is supra-simplicial set in Theorem \ref{cage} is essential:  not any subset of nodes of the cardinality $\#\mathsf A$ from a $d^{\{n\}}$-cage imposes independent relations on the set of homogeneous polynomials of degree $d$ in $n+1$ variables! 

For example, in a $4\times 4$-cage, $\#\mathsf A = 13$. However, if $\mathsf B$ is the complement to the set of four nodes $\mathsf C := \{p_{42}, p_{43}, p_{44}\}$, then not every curve of degree $4$ that contains $\mathsf B$ will contain $\mathsf C$. In fact,  $\mathsf B$ is contained in the union of three red and one blue lines from the cage; they all miss $\mathsf C$.
\hfill $\diamondsuit$
\smallskip

\noindent {\bf Example 2.5.} Consider any curve $C$ in $\P^3$, given by homogeneous polynomial equations of degree $\leq 3$ (typically, $C$ is of degree $9$). If $C$ passes through $17$ nodes of a supra-simplicial set $\mathsf A$ of nodes of a $3^{\{3\}}$-cage, then it passes through all the $27$ nodes of the cage.

A similar conclusion holds for any surface of degree $3$ in $\P^3$ that passes through the $17$ nodes from $\mathsf A$. \hfill $\diamondsuit$
\smallskip

\begin{corollary}\label{d_exactly} Consider a subvariety $V \subset \P^n$, given by one or several homogeneous polynomial equations of degrees $\leq d$. If $V$ contains all the nodes from a supra-simplicial set $\mathsf A$ (of cardinality $C_{d+n}^n -  n$) in a $d^{\{n\}}$-cage $\mathsf K \subset \P^n$, then all the polynomials that define $V$ are exactly of degree $d$.
\end{corollary}

\begin{proof} By Theorem \ref{main_th}, if a homogeneous polynomial $P$ of degree less than $d$, which vanishes at $V$, also vanishes at the simplicial set $\mathsf T \subset \mathsf A$ of the $d^{\{n\}}$-cage $\mathsf K$, then $P = 0$ identically. Thus $\deg P = d$, provided that $P$ is nontrivial.
\end{proof}

In this paper, we will use the following ``working  definition" of complete intersections.

\begin{definition}\label{complete_int}
A purely $(n-k)$-dimensional Zariski closed subset $V \subset \P^n$ 
over a field $\A$ of characteristic zero is a {\sf complete intersection} of the type $(d_1, \ldots, d_k)$, if $V$ is an intersection of hypersurfaces $H_1, \ldots, H_k$ in $\P^n$ of degrees $d_1, \ldots, d_k$ such that, at a generic (nonsingular) point $x$ of each component of $V$, the tangent hyperspaces $\{T_xH_j\}_{j\in [1, k]}$ are in general position in the ambient tangent space $T_x\P^n$.  \hfill $\diamondsuit$
\end{definition}
\smallskip

The next proposition was stated  by the reviewer of this paper as a natural generalization of some claims from our Theorem \ref{smooth}. It puts this theorem in a larger context.

\begin{proposition}\label{TH_A} Suppose that $W \subset \P^n$  is a smooth 
complete intersection of $r \leq n$ hypersurfaces of degree $d$ and that $V \supset  W$ is an algebraic subset of $\P^n$ that is (set-theoretically) an intersection of hypersurfaces of degree  $d$. Then is $V$ is a complete intersection of hypersurfaces of degree $d$, and $V$ is smooth along $W$.  
\end{proposition}

\begin{proof} Let homogeneous polynomials $g_1, \ldots , g_r \in \A[x_0, \ldots , x_n]$ of  degree $d$ be generators of the zero ideal $\mathcal I(W)$ of the complete intersection $ W$ of codimension $r$.  In particular, they are linearly independent vectors of the $d$-graded portion of the ring $\A[x_0, \ldots , x_n]$. Let $V$ be the zero set of homogeneous degree $d$ polynomials $f_1, \ldots , f_s \in \A[x_0, \ldots , x_n]$. Since $V \supset  W$,  each $f_j \in  \mathcal I(W)$. Since $f_j$ is of the degree $d$, we get $f_j = \sum_k c_{kj}\, g_k$,  a linear combination of $g_k$'s over $\A$. We may choose a maximal subset of the set $\{f_1, \ldots , f_s\}$, comprised of linearly independent elements. By linear algebra, $s \leq r$. Abusing notations, we denote  these new elements by $f_1, \ldots , f_t$. To show that $V$ is a complete intersection, we need to show that the codimension of  $V$ in $\P^n$ is $t$.  At a smooth point $p \in W$, in an affine chart that contains $p$, we replace the homogeneous polynomials $g_1, \ldots , g_r$ by their non-homogeneous representations $\tilde g_1, \ldots , \tilde g_r$ in the affine coordinates. Then the differentials $d\tilde g_1, \ldots , d\tilde g_r$ at $p$ (that span the normal cotangent to $ W$ bundle) are linearly independent. Thus, by linear algebra, $d\tilde f_1, \ldots , d\tilde f_t$ are linearly independent at $p$ as well.  Therefore $V$ is smooth at $p \in  W$ and $\text{codim}(V, \P^n) = t$. 
\end{proof}

We choose the  node set $\mathsf N \subset \mathsf K$ for the role of the complete intersection $W$ from Proposition \ref{TH_A}. With this choice, Proposition \ref{TH_A},  being combined with Theorem \ref{main_th}, implies directly  Theorem \ref{smooth} below. 

\begin{theorem}\label{smooth} Let  $V \subset \C\P^n$ be a subvariety of codimension $s$, given by one or several homogeneous polynomial equations of degrees $\leq d$. 

If $V$ contains all the nodes from a supra-simplicial set $\mathsf A$ of a $d^{\{n\}}$-cage $\mathsf K \subset \C\P^n$, then $V$ is a complete intersection of the multi-degree $(\underbrace{d, \dots , d}_{s})$, which is smooth at each node of the cage $\mathsf K$. Thus, $\deg V = d^{n-\dim V}$. \hfill $\diamondsuit$
\end{theorem}

In turn, Theorem \ref{smooth} forces the following obvious logical conclusion. 

\begin{corollary}\label{not_complete} If a variety $V \subset \C\P^n$, given by homogeneous polynomials of degrees $\leq d$, is not a complete intersection, then it cannot be trapped in any $d^{\{n\}}$-cage in $\P^n$. \hfill $\diamondsuit$
\end{corollary}

Perhaps, Theorem \ref{smooth} and Corollary \ref{not_complete} are valid also over the base field $\R$.\smallskip

\noindent {\bf Example 2.6.} Since the twisted cubic curve $\mathcal C: [s:  t] \to [s^3: s^2t : st^2: t^3]$ is not a complete intersection in $\C\P^3$, by Corollary \ref{not_complete}, $\mathcal C$ does not contain the nodes of any $3^{\{3\}}$-cage $\mathsf K$ in $\C\P^3$, or even the nodes from a supra-simplicial set $\mathsf A \subset \mathsf K$. \hfill $\diamondsuit$
\smallskip

\noindent {\bf Example 2.7.}\label{ex_Q} 
Despite looking diverse, all the figures in this paper depict varieties, attached to the nodes of $d^{\{n\}}$-cages $\mathsf K(\mathbf Q)$ that are produced following a very simple recipe. It starts with a small set $\mathbf Q \subset \A^n$ of ``nodes in the making" and uses the product structure in $\A^n$ as follows.   

Consider $d$ points $q_1, \dots, q_d \in \A^n$ such that, for each coordinate function $z_j: \A^n \to \A$, their $z_j$-coordinates are distinct. Let us denote by $\A(n, d)$ the space of such configurations $\mathbf Q := (q_1, \dots, q_d)$. Then each $\mathbf Q \in \A(n, d)$ produces a $d^{\{n\}}$-cage $\mathsf K(\mathbf Q) \subset \A^n$, formed by the hyperplanes $\big\{H_{j, i} := z_j^{-1}(z_j(q_i))\big\}_{j \in [1, n],\, i \in [1, d]}$. By Theorem \ref{main_th} and Theorem \ref{smooth}, for any $s \leq n$, the cage $\mathsf K(\mathbf Q)$ supports the family of varieties $X$ of the multi-degree $(\underbrace{d, \dots, d}_{s})$ and dimension $n-s$ that contain the node set $\mathsf N(\mathbf Q)$ of $\mathsf K(\mathbf Q)$.  By Theorem \ref{unique_hyper} below, the family is parametrized by points of the Grassmanian $\mathsf{Gr}_\A(n, n-s)$.
\smallskip

Over $\C$, we can enhance this cage construction. Consider the complex Vi\`{e}te map $\Sigma: \C^n \to \mathsf{Sym}^n\C \approx \C^n$, given by the elementary symmetric polynomials in $z_1, \dots, z_n$. It takes the ``roots" $z_1, \dots, z_n \in \C$ to the coefficients of the monic polynomial $\prod_{j=1}^n (x - z_j)$ in the variable $x$. 

We denote by $\mathcal D$ the hypersurface in $\mathsf{Sym}^n\C$, formed by the $x$-polynomials with multiple roots. It is called the {\sf discriminant variety}. Remarkably, the $\Sigma$-images of the hyperplanes $\{H_{j, i} \subset \C^n \}$ are hyperplanes, \emph{tangent} to $\mathcal D$; moreover, the normal vector to $\Sigma(H_{j, i})$, whose $n^{th}$ coordinate is $1$, has its $(n-1)^{st}$ coordinate equal to $z_j(q_i)$ (\cite{K2}, Corollary 6.1)!  

Therefore, $\{\Sigma(H_{j, i})\}_{j, i}$ form a new $d^{\{n\}}$-cage $\Sigma(\mathsf K(\mathbf Q))$ in $\mathsf{Sym}^n\C \approx \C^n$, whose hyperplanes are tangent to the discriminant variety $\mathcal D$. The nodes of $\Sigma(\mathsf K(\mathbf Q))$ reside in $\C^n$. Via the tangency property, the cage $\Sigma(\mathsf K(\mathbf Q))$ is completely determined by the configuration $\Sigma(\mathbf Q)$ of $d$ points in $\C^n \setminus \mathcal D$, since any point $p \in \C^n \setminus \mathcal D$ belongs to exactly $n$ hyperplanes that are tangent to $\mathcal D$ \cite{K2}. As a result, any generic (that is, of the form $\Sigma(\C(n, d))$) configuration $\mathbf P$ of points $p_1, \dots, p_d \in \C^n \setminus \mathcal D$ produces a $d^{\{n\}}$-cage $\mathsf K(\mathbf P)$ in $\C^n$, whose hyperplanes are alined with the tangent cones of $\mathcal D$. Again, for any $s \leq n$, the cage $\mathsf K(\mathbf P)$ supports a family of complex varieties $V$ of the multi-degree $(\underbrace{d, \dots, d}_{s})$ and dimension $n-s$ that contain the node set $\mathsf N(\mathbf P)$ of $\mathsf K(\mathbf P)$.  By Theorem \ref{unique_hyper}, the family is parametrized by points of the Grassmanian $\mathsf{Gr}_\C(n, n-s)$.\smallskip

Thus we got an effective device for producing varieties in cages.  A configuration $\mathbf Q \in \C(n,d)$ or a configuration $\mathbf P \in \Sigma(\C(n,d))$, together with a choice of a $(n-s)$-dimensional affine subspace $\tau \subset \C^n$ at a point $q_1\in \mathbf Q$, or a $(n-s)$-dimensional affine subspace $\tilde\tau \subset \C^n$ at a point $p_1 \in \mathbf P$, produce unique varieties $X(\mathbf Q, \tau) \subset \C^n$ and  $Y(\mathbf P, \tilde\tau) \subset \C^n_{\mathsf{coef}}$ of the dimension $n-s$ that are attached to the nodes of the two cages, respectively. \smallskip

Over the real numbers, the outcome is similar, if we consider only the chamber $\mathcal C$ in the space $\R^n_{\mathsf{coef}}$ of monic real polynomials with all real roots; $\mathcal C$ is one of many chambers in which the real discriminant hypersurface $\mathcal D_\R$ divides $\R^n_{\mathsf{coef}}$. So, over $\R$, the cage-generating configuration $\mathbf P$ must be chosen in the chamber $\mathcal C$.
\smallskip

The construction $(\mathbf Q, \tau) \Rightarrow X(\mathbf Q, \tau)$ has one pleasing property: if the configuration $\mathbf Q$ consists of $d$ points with all the coordinates in $\Z$ or $\Q$, then the variety $X(\mathbf Q, \tau)$ contains at least $d^n$ integral or rational points. Since the Vi\`{e}te map $\Sigma$ is given by elementary symmetric polynomials with integer coefficients, the same property holds for any variety $Y(\Sigma(\mathbf Q), \tilde\tau))$ that is attached to the nodes of the cage $\mathsf K(\Sigma(\mathbf Q)) \subset \mathsf{Sym}^n\C$. 
\hfill $\diamondsuit$
\smallskip

\begin{remark}
\emph{Recall that, thanks to the Lefschetz Hyperplane Theorem (see Corollary 7.3 and Theorem 7.4 in \cite{Mi}), the topology of complete intersections is very special. In particular, by a theorem of Thom, the diffeomorphism type of a smooth complete intersection $X$ over $\C$ is determined by its dimension and its multi-degree $(d_1, \ldots, d_k)$ \cite{LW}.  Therefore, the combinatorics of a supra-simplicial set imposes strong restrictions on the homology groups, homotopy types, and characteristic classes of the varieties inscribed in a given cage.} \hfill $\diamondsuit$
\end{remark}
\smallskip

Let us consider a $n$-dimensional polyhedron $\mathcal P$ in $\R^n$, whose combinatorics is modeled after the combinatorics of a $n$-cube. The opposite faces of $\mathcal P$ are labeled with the same color; so the total pallet has $n$ colors. 
We wish to place the vertices  of $\mathcal P$ on a given variety $V \subset \R^n$ that is defined as the zero set of several \emph{quadratic} polynomials (think about $V$ as being an ellipsoid or a hyperboloid). The next corollary testifies that in order to accomplish this task, one needs to place just few vertices of $\mathcal P$ on $V$, the rest of the vertices  will reside in $V$ automatically.  Actually, the following direct corollary of Theorem \ref{smooth} makes sense over any infinite field $\A$.\smallskip

Recall that \emph{any} complex projective variety may be given as an intersection of {\sf quadrics} in an appropriate projective space \cite{M1}.

\begin{corollary}{\bf (Varieties  in the Cube Cage)} 
\noindent Let a variety $V \subset \P^n$ be given by homogeneous polynomials of degree $2$ and contains all $\frac{1}{2}(n^2 + n + 2)$ nodes of a supra-simplicial set $\mathsf A$ in a $2^{\{n\}}$-cage $\mathsf K \subset \P^n$.  \smallskip

Then $V$ is a complete intersection of degree $2^s$, where $s = n - \dim(V)$. Moreover, $V$ contains all $2^n$ nodes of $\mathsf K$ and is smooth in their vicinity. \hfill $\diamondsuit$
\end{corollary}

\noindent {\bf Example 2.9.} If a smooth curve $C \subset \P^3$ is given by two homogeneous quadratic forms and contains $7$ nodes of a $2^{\{3\}}$-cage $\mathsf K \subset \P^3$, then it contains the $8^{th}$ node of the cage. Moreover, $C$ is a complete intersection of the multi-degree $(2, 2)$. 
In fact, such a curve $C$ is {\sf elliptic} (i.e., smooth and of genus $1$). 
\hfill $\diamondsuit$
\smallskip

Proposition \ref{TH_B} below is another claim, formulated by the reviewer. It frames well our Corollary \ref{unique_hyper}. 

\begin{proposition}\label{TH_B} Suppose that $W \subset \P^n$  is a smooth 
complete intersection of $r \leq n$ hypersurfaces of degree $d$. Pick a point $p \in  W$ and a linear subspace $\tau_p \subset T_p\P^n$ of dimension $\delta > n-r$ such that $\tau_p \supset T_p W$. 
Then there exists a unique complete intersection $V \subset \P^n$ of $n- \delta$ hypersurfaces of degree $d$ such that $V \supset W$ and $T_pV = \tau_p$.  
\end{proposition}

\begin{proof} As in the proof of Proposition \ref{TH_A}, let homogeneous polynomials $g_1, \ldots , g_r \in\hfill\break  \A[x_0, \ldots , x_n]$ of degree $d$ be generators of the zero ideal $\mathcal I(W)$ of a smooth complete intersection $W$ of the codimension $r$. We denote by $L(W)$ the linear space spanned over $\A$ by $g_1, \ldots , g_r$. In an affine chart that contains the given point $p \in W$, we replace $g_1, \ldots , g_r$ by their expressions as polynomials $\tilde g_1, \ldots , \tilde g_r$ in the affine coordinates. Since $W$ is smooth at $p$, the differentials  $d\tilde g_1, \ldots , d\tilde g_r$ at $p$ vanish exactly at the tangent subspace $T_p W$ of codimension $r$. We consider a linear subspace $\mu^\ast(\tau_p)$ of $\nu^\ast_p := \mathsf{Span}_{\A}\{d\tilde g_1|_p, \ldots , d\tilde g_r|_p\}$ which consists of $1$-forms that vanish at the given vector space $\tau_p \supset T_p W$. Let $d\tilde h_1, \ldots , d\tilde h_{n-\delta}$ be a basis of the vector space $\mu^\ast(\tau_p)$. Then each $d\tilde h_j = \sum_k a_{jk}\, d\tilde g_k$ for some $a_{jk} \in \A$. Consider the polynomials $\{\tilde h_j := \sum_k a_{jk}\, \tilde g_k\}_{j \in [1, n-\delta]}$. Denote by $\{h_j\}_{j \in [1, n-\delta]}$ their homogeneous representatives. Then we define $V$ by $n-\delta$ homogeneous equations $\{h_j = 0\}_{j \in [1, n-\delta]}$, where $\{h_j\}$ are linearly independent elements in the $d$-graded portion of $\A[x_0, \ldots , x_n]$. By its construction, $V$ has $\tau_p$ for its tangent space at $p$.  So its dimension is $\delta$. Therefore $V$ is a complete intersection which is smooth at $p$.\smallskip

The same kind of arguments, based only on linear algebra, validates the claim that $V$ is unique among complete intersections of the multi-degree  $(d, \ldots , d)$ and with $\tau_p$ for a tangent space at some point $p \in W$. 
\end{proof}

By taking the node set $\mathsf N$ for the role of the complete intersection $W$ in Proposition \ref{TH_B} and applying Theorem \ref{smooth}, we get instantly the following claim.

\begin{corollary}\label{unique_hyper} Consider a $d^{\{n\}}$-cage $\mathsf K \subset \P^n$ and a vector subspace $\tau_p$ of dimension $n-s$ in the tangent space $T_p(\P^n)$, where $p$ is a node of $\mathsf K$. Then there exists a unique complete intersection $V \subset \P^n$ of the multi-degree $(\underbrace{d, \dots , d}_{s})$ and of dimension $n-s$ that contains all the nodes of  $\mathsf K$ and whose tangent space $T_p(V) = \tau_p$. 
 \smallskip

As a result, any supra-simplicial node set $\mathsf A \subset \mathsf K$ and a $(n-s)$-dimensional subspace  $\tau_p \subset T_p(\P^n)$,\footnote{equivalently, a point in the Grassmanian $\mathsf{Gr}_\A(n, n-s)$} where $p \in \mathsf N$, determines such a  variety $V$ and the \emph{distribution} of $(n-s)$-subspaces $\tau_V$ in $T(\P^n)|_{\mathsf N}$ it produces. So the cage $\mathsf K$, with the help of the inscribed $V$'s, defines canonically a ``diagonal" embedding of Grassmanians  $$\Delta_\mathsf K:\; \mathsf{Gr}_\A\big(T_p(\P^n),\, n-s\big) \longrightarrow \prod_{q \in \mathsf N \setminus p} \mathsf{Gr}_\A\big(T_q(\P^n),\, n-s\big).$$
\hfill $\diamondsuit$
\end{corollary} 

We may give a slightly different interpretation to Corollary \ref{unique_hyper} by viewing the Grassmanian $\mathsf{Gr}_\A(n, n-s)$ as a moduli space of varieties in a given $d^{\{n\}}$-cage $\mathsf K \subset \P^n$.  Consider a subvariety $\mathcal E(\mathsf K, p)$ of $\P^n \times \mathsf{Gr}_\A(n, n-s)$
which depends on the cage $\mathsf K \subset \P^n$ and its preferred node $p$. By definition, $\mathcal E(\mathsf K, p) = \{(x, \tau) | \; x \in V(\tau),\; \tau \in \mathsf{Gr}_\A(n, n-s)\}$, where $V(\tau) \subset \P^n$ is the complete intersection  of dimension $n-s$ and multi-degree $(\underbrace{d, \dots , d}_{s})$ that contains all the nodes of $\mathsf K$ and has $\tau$ as its tangent space $T_p(V(\tau))$ at $p$. Then the obvious projection $\pi: \P^n \times \mathsf{Gr}_\A(n, n-s) \to \mathsf{Gr}_\A(n, n-s)$
 gives rise to a surjective regular map of algebraic varieties $\pi: \mathcal E(\mathsf K, p) \to \mathsf{Gr}_\A(n, n-s)$ whose fiber over a point $\tau \in \mathsf{Gr}_\A(n, n-s)$ is the complete intersection $V(\tau)$. 

The subvarieties of the Grassmanian  $\mathsf{Gr}_\A(n, n-s)$ that consist of points $\tau$ such that the fiber $\pi^{-1}(\tau)$ is \emph{singular} or \emph{reducible} (and thus singular) deserve a separate research.  
 \smallskip

The reader may be entertained by one special case of Corollary \ref{unique_hyper}, where the variety $V$ is a complete intersection of the multi-degree $(\underbrace{d, \dots , d}_{n-1})$, a curve with a given tangent line $\ell$ at one of the nodes of the cage to which $V$ is attached. In this ``Cage Croquet" game, the nodes represent the gates, and the curve represents the desired trajectory of the ball, passing through all the gates. One can aim in any direction $\ell$ from any gate, and the ``right trajectory" $V$ of the multi-degree $(\underbrace{d, \dots , d}_{n-1})$ will pass through all the gates.  \smallskip

Let us glance now at $3$-dimensional cages and the polyhedral surfaces that have their vertices among the nodes of these cages. \smallskip

A {\sf tricolored polyhedral surface}  $\Sigma \subset \R^3$ is a surface whose faces are flat polygons, colored with three colors. We say that a vertex $v$ of $\Sigma$ is {\sf trivalent} if exactly three distinctly colored faces join at $v$.  A tricolored polyhedral surface is {\sf trivalent} if all its vertices are. Finally, a {\sf perfect} trivalent polyhedral surface is a trivalent tricolored polyhedral surface with equal number of faces, colored with each of the three colors. A surface of a cube is an example of a perfect trivalent polyhedral surface. 

We notice that a generic perfect trivalent polyhedral surface with $3d$ faces determines a $d^{\{3\}}$-cage in the space. For example, a generic union of $k$ tricolored cubes in $\R^3$ is a perfect trivalent polyhedron that gives rise to a $(2k)^{\{3\}}$-cage in $\R^3$. 
\smallskip

In view of these observations, Corollary \ref{unique_hyper} leads to the following claim.

\begin{corollary} Let $\Sigma \subset \R^3$ be a perfect trivalent polyhedral surface with $3d$ faces that generates a $d^{\{3\}}$-cage $\mathsf K_\Sigma$ in $\R^3$. Given a plane $\tau$ through one of vertices $v \in \Sigma$, there exists a unique affine algebraic surface $S \subset \R^3$ of degree $d$ such that: 
\begin{itemize}
\item all the verticies of $\Sigma$ lie on $S$ (i.e., $\Sigma$ is inscribed in $S$), 

\item  $S$ contains all the nodes of the $d^{\{3\}}$-cage $\mathsf K_\Sigma$, 

\item  $S$ is tangent to the plane $\tau$ at $v$ and is smooth in the vicinity of all verticies of $\Sigma$.  

\hfill $\diamondsuit$
\end{itemize}
\end{corollary}

\noindent {\bf Example 2.10.}
The surface $\Sigma$ of a surface of a tricolored cube with three quadrangular wormholes that connect pairs of similarly colored opposite faces ($\Sigma$ is a surface of genus $3$) has  $18 = 6 + 3\times 4$ faces (6 of which are not simply-connected polygons). It can be inscribed in an algebraic surface $S$ of degree $6 = 18/3$. In addition to the 32 vertices of $\Sigma$, lying on $S$, the rest of the nodes (numbering 184) of the $6^{\{3\}}$-cage $\mathsf K_\Sigma$ also belongs to $S$. Such a surface $S$ with a prescribed tangent plane $\tau$ at one vertex of $\Sigma$ is unique.  \hfill $\diamondsuit$
\bigskip

{\it Acknowledgments: } The earlier version of this paper contained a generalization  of Theorem \ref{main_th} for cages on complex projective varieties. My proof of this generalization contained a basic irreparable mistake. I am very grateful to the reviewer who has found the mistake and came with valuable suggestions for improving the entire presentation.


\begin{thebibliography}{30}

\bibitem [B]{B} Bacharach, I., {\it Uber den Cayley'schen Schnittpunktsatz}, Math. Ann. 26 (1886), 275-299.

\bibitem [C]{C} Cayley, A., {\it On the Intersection of Curves}, serialized in Cambridge Math. J., 25-27, and published by Cambridge University Press, Cambridge, 1889.

\bibitem [C1]{C1} Cayley, A., On the triple tangent planes of surfaces of the third order, Cambridge and Dublin Math. J., 4: 118-138, (1849).

\bibitem [Ch]{Ch} Chasles, M., {\it Traite de sections conique}, Gauthier-Villars, Paris, 1885.

\bibitem [Cl]{Cl} Clebsch, A., {\it Ueber die Anwendung der quadratischen Substitution auf die Gleichungen 5ten Grades und die geometrische Theorie des ebenen F\"{u}nfseits}, Mathematische Annalen, 4 (2): 284-345, (1871).

\bibitem [DGO]{DGO}  Davis, E. D., Geramita, A.V., and Orecchia, F., {\it Gorenstein algebras and Cayley-Bacharach theorem}, Proceedings Amer. Math. Soc. 93 (1985), 593-597. 

\bibitem [EGH]{EGH} Eisenbud, D., Green, M., and Harris, J., {\it Cayley-Bacharach theorems and conjectures}, Bull. Amer. Math. Soc. 33 (1996),  295-324.

\bibitem [F]{F} Fulton, W., {Introduction to Intersection Theory in Algebraic Geometry}, Regional Conference Series in Mathematics, Number 54, American Mathematical Society, Providence, 1983.

\bibitem[GHS]{GHS} Geramita, A. V., Harima, T., and Shin, Y.S., {\it Extremal point sets and Gorenstein ideals}, Adv. Math 152 (2000) 78-119.

\bibitem[GHS1]{GHS1} Geramita, A. V., Harima, T., and Shin, Y.S., {\it An alternative to the Hilbert function for the ideal of a finite set of points in $\P^n$}, Illinois J. 45 (2001) 1-23.

\bibitem[GHS2]{GHS2} Geramita, A. V., Harima, T., and Shin, Y.S., {\it Decompositions of the Hilbert Function of a Set of Points in $\P^n$}, Canad. J. Math. 53 (2001) 925-943. 

\bibitem[H]{H} Hartshorne, R., {\it Algebraic Geometry}, Springer-Verlag, New York, 1977.

\bibitem [K]{K}  Katz, G., {\it Curves in Cages: An Algebro-Geometric Zoo}, The American Mathematical Monthly, 113 : 9 (2006), 777-791.

\bibitem [K2]{K2}  Katz, G., {\it How Tangents Solve Algebraic Equations, or a Remarkable Geometry of Discriminant Varieties}, Expositiones Math., 21 (2003), 219-261.

\bibitem[LW]{LW}  Libgober, A. S.,  Wood, J. W., {\it Differentiable structures on complete intersections. I}, Topology 21 (1982), 469-482. 

\bibitem [Ki]{Ki} Kirwan, M. {\it Complex Algebraic Curves, London Mathematical Society, Student Texts 23}, Cambridge University Press, Cambridge, 1992.

\bibitem [Kl]{Kl} Klein, F., {\it \"Uber Fl\"{a}chen dritter Ordnung}, Math. Ann., 6 (1873), 551-581.

\bibitem [M]{M} Mumford, D., {\it Algebraic Geometry I, Complex Projective Varieties}, Springer-Verlag, Berlin Heidelberg New York, 1976.

\bibitem [M1]{M1} Mumford, D., {\it Varieties Defined by Quadratic Equations}, Centro Internazionale Matematico Estivo, Corso tenuto a Varena, Septembre 1969.

\bibitem [Mi]{Mi} Milnor, J., {\it Morse Theory}, Annals of Mathematical Studies, no. 51, Princeton University Press, Princeton, 1963.

\bibitem [R]{R} Ried, M. {\it Undergraduate Algebraic Geometry, London Mathematical Society, Student Texts 12.\/},  Cambridge University Press, Cambridge, 1998.

\end{thebibliography}
\end{document}